\documentclass{amsart}
\usepackage{amsmath}
\usepackage{amsthm}
\usepackage{amssymb}
\usepackage{amscd}
\usepackage{xypic}
\xyoption{all}
\usepackage{hyperref}

\let\fullref\autoref
\def\makeautorefname#1#2{\expandafter\def\csname#1autorefname\endcsname{#2}}
\makeautorefname{section}{Section}
\makeautorefname{theorem}{Theorem}
\makeautorefname{proposition}{Proposition}
\makeautorefname{lemma}{Lemma}
\makeautorefname{corollary}{Corollary}
\makeautorefname{definition}{Definition}
\makeautorefname{remark}{Remark}
\makeautorefname{example}{Example}

\begin{document}

%%%%%%%%%%%%%%%%%%%%%%%%%%%%%%%%%%%%%%%%%%%%%%%%%%%%%%%%%%%%%%%%%%%%%%%%
%%%%%%%%%%%%%%%%%%%%%%%%%%%%%%%%%%%%%%%%%%%%%%%%%%%%%%%%%%%%%%%%%%%%%%%%
\title[Vector bundles over classifying spaces of $p$--local finite groups]
{Vector bundles over classifying spaces of $p$--local finite groups and Benson--Carlson duality}
%%%%%%%%%%%%%%%%%%%%%%%%%%%%%%%%%%%%%%%%%%%%%%%%%%%%%%%%%%%%%%%%%%%%%%%%
%%%%%%%%%%%%%%%%%%%%%%%%%%%%%%%%%%%%%%%%%%%%%%%%%%%%%%%%%%%%%%%%%%%%%%%%

\author{Jos\'e Cantarero}
\address{\hfill \break Consejo Nacional de Ciencia y Tecnolog\'ia \\ 
\hfill \break Centro de Investigaci\'on en Matem\'aticas, A.C. Unidad M\'erida \\ 
\hfill \break Parque Cient\'ifico y Tecnol\'ogico de Yucat\'an \\ 
\hfill \break Carretera Sierra Papacal--Chuburn\'a Km 5.5 \\ 
\hfill \break M\'erida, YUC 97302 \\ 
\hfill \break Mexico.}
\email{cantarero@cimat.mx}

\author{Nat\`alia Castellana}
\address{\hfill \break Departament de Matem\`atiques\\ 
         \hfill \break Universitat Aut\`onoma de Barcelona and BGSMath\\ 
         \hfill \break Edifici Cc\\ 
         \hfill \break E--08193 Bellaterra\\ 
         \hfill \break Spain.}
\email{natalia@mat.uab.cat}         
         
\author{Lola Morales}
\address{\hfill \break IES Clara Campoamor\\ 
         \hfill \break Av. de Alcorc\'on, 1\\  
         \hfill \break E--28936 M\'ostoles\\ 
         \hfill \break Spain.}

\subjclass[2010]{55R35, (primary), 20D20, 20C20 (secondary)}
\keywords{$p$--local, fusion system, Benson--Carlson duality}

%%%%%%%%%%%%%%%%%%%%%%%%%%%%%%%%%%%%

\newcommand{\B}{{\mathbb B}}
\newcommand{\C}{{\mathbb C}}
\newcommand{\Ca}{{\mathcal C}}
\newcommand{\F}{{\mathbb F}}
\newcommand{\Ff}{{\mathcal{F}}}
\newcommand{\K}{\mathbb{K}}
\newcommand{\Ll}{\mathcal{L}}
\newcommand{\m}{\mathfrak{m}}
\newcommand{\Or}{{\mathcal O}}
\newcommand{\Q}{{\mathbb Q}}
\newcommand{\Z}{{\mathbb Z}}

\newcommand{\Aut}{\operatorname{Aut}\nolimits}
\newcommand{\Co}{\operatorname{Co}\nolimits}
\newcommand{\End}{\operatorname{End}\nolimits}
\newcommand{\Hom}{\operatorname{Hom}\nolimits}
\newcommand{\id}{\operatorname{id}\nolimits}
\newcommand{\im}{\operatorname{Im}\nolimits}
\newcommand{\Inj}{\operatorname{Inj}\nolimits}
\newcommand{\Inn}{\operatorname{Inn}\nolimits}
\newcommand{\Iso}{\operatorname{Iso}\nolimits}
\newcommand{\Map}{\operatorname{Map}\nolimits}
\newcommand{\Mod}{\operatorname{Mod}\nolimits}
\newcommand{\Mor}{\operatorname{Mor}\nolimits}
\newcommand{\op}{\operatorname{op}\nolimits}
\newcommand{\Out}{\operatorname{Out}\nolimits}
\newcommand{\pt}{\operatorname{pt}\nolimits}
\newcommand{\reg}{\operatorname{reg}\nolimits}
\newcommand{\res}{\operatorname{res}\nolimits}
\newcommand{\Rep}{\operatorname{Rep}\nolimits}
\newcommand{\Set}{\operatorname{Set}\nolimits}
\newcommand{\Sol}{\operatorname{Sol}\nolimits}
\newcommand{\Spin}{\operatorname{Spin}\nolimits}
\newcommand{\tr}{\operatorname{tr}\nolimits}
\newcommand{\Top}{\operatorname{Top}\nolimits}
\newcommand{\Vect}{\operatorname{Vect}\nolimits}

\newcommand{\higherlim}[2]{\displaystyle\setbox1=\hbox{\rm lim}
	\setbox2=\hbox to \wd1{\leftarrowfill} \ht2=0pt \dp2=-1pt
	\setbox3=\hbox{$\scriptstyle{#1}$}
	\def\test{#1}\ifx\test\empty
	\mathop{\mathop{\vtop{\baselineskip=5pt\box1\box2}}}\nolimits^{#2}
	\else
	\ifdim\wd1<\wd3
	\mathop{\hphantom{^{#2}}\vtop{\baselineskip=5pt\box1\box2}^{#2}}_{#1}
	\else
	\mathop{\mathop{\vtop{\baselineskip=5pt\box1\box2}}_{#1}}%
	\nolimits^{#2}
	\fi\fi}
	
\newcommand{\hocolim}[2]{\displaystyle\setbox1=\hbox{\rm hocolim}
	\setbox2=\hbox to \wd1{\rightarrowfill} \ht2=0pt \dp2=-1pt
	\setbox3=\hbox{$\scriptstyle{#1}$}
	\def\test{#1}\ifx\test\empty
	\mathop{\mathop{\vtop{\baselineskip=5pt\box1\box2}}}\nolimits^{#2}
	\else
	\ifdim\wd1<\wd3
	\mathop{\hphantom{^{#2}}\vtop{\baselineskip=5pt\box1\box2}^{#2}}_{#1}
	\else
	\mathop{\mathop{\vtop{\baselineskip=5pt\box1\box2}}_{#1}}%
	\nolimits^{#2}
	\fi\fi}

% p-lfgs declarations

\newcommand{\pcom}{^\wedge_p}

%Declarations
\theoremstyle{plain}
\newtheorem{theorem}{Theorem}[section]
\newtheorem{proposition}{Proposition}[section]
\newtheorem{corollary}{Corollary}[section]
\newtheorem{lemma}{Lemma}[section]

\makeatletter
\let\c@lemma=\c@theorem
\makeatother

\makeatletter
\let\c@proposition=\c@theorem
\makeatother

\makeatletter
\let\c@corollary=\c@theorem
\makeatother

\newtheorem*{introtheorem}{Theorem} 
\theoremstyle{definition}

\newtheorem{definition}{Definition}[section]
\newtheorem{remark}{Remark}[section]
\newtheorem{example}{Example}[section]

\makeatletter
\let\c@definition=\c@theorem
\makeatother

\makeatletter
\let\c@remark=\c@theorem
\makeatother

\makeatletter
\let\c@example=\c@theorem
\makeatother

\begin{abstract}
In this paper we obtain a description of the Grothendieck group of complex vector bundles over the classifying space 
of a $p$--local finite group $(S,\Ff,\Ll)$ in terms of representation rings of subgroups of $S$. We also prove a stable 
elements formula for generalized cohomological invariants of $p$--local finite groups, which is used to show the existence 
of unitary embeddings of $p$--local finite groups. Finally, we show that the augmentation $C^*(|\Ll|\pcom;\F_p)\to \F_p$ 
is Gorenstein in the sense of Dwyer--Greenlees--Iyengar and obtain some consequences about the cohomology ring of $|\Ll|\pcom$.
\end{abstract}

\maketitle

%%%%%%%%%%%%%%%%%%%%%%%%%%%%%%%%%%%%%%%%%%%%%%%%%%%%%%%%%%%%%%%%%%%%%%%%%%%%%%%%
\section{Introduction}
\label{intro}
%%%%%%%%%%%%%%%%%%%%%%%%%%%%%%%%%%%%%%%%%%%%%%%%%%%%%%%%%%%%%%%%%%%%%%%%%%%%%%%%

Let $\K(X)$ be the Grothendieck group of the monoid of complex vector bundles 
over $X$. When $P$ is a finite $p$--group, Dwyer and Zabrodsky \cite{DZ} showed 
that there is an isomorphism $\K(BP) \cong R(P)$, where $R(P)$ denotes the representation 
ring of $P$. A few years later, Jackowski and Oliver proved in \cite{JO} that if $G$ is a 
compact Lie group, there is an isomorphism
\[ \K(BG) \stackrel{\cong}{\longrightarrow} \higherlim{}{} R(P) \]
where the inverse limit is taken over all $p$--toral subgroups (for all primes) with respect to inclusion and conjugation. 
Note that this statement has a clear `local' flavour in the sense that this object depends only on a family of subgroups 
of $G$ and conjugacy relations among them.

The notions of $p$--local finite groups \cite{BLO1} and $p$--local compact groups \cite{BLO2} introduced by Broto, Levi
and Oliver model the $p$--local
information of finite and compact Lie groups, respectively. Hence one would expect to have an analogous description
for the Grothendieck group of finite-dimensional complex vector bundles over their classifying spaces. In this paper
we focus on $p$--local finite groups. Partial results about $p$--local compact groups in this direction were obtained 
by Cantarero and Castellana in \cite{CC}.

Recall that a $p$--local finite group is a triple $(S,\Ff,\Ll)$, where $S$ is a finite $p$--group, $\Ff$ is a saturated fusion system over
$S$ and $\Ll$ is a centric linking system associated to $\Ff$. The classifying space of $(S,\Ff,\Ll)$ is $ | \Ll | \pcom $. More details 
can be found in \fullref{section:preliminaries} of this article. Our generalization in this context is the
following theorem.

\begin{introtheorem}
Given a $p$--local finite group $(S,\Ff,\Ll)$, restriction to the subgroups of $S$ 
gives an isomorphism 
\[ \K(|\Ll|\pcom) \stackrel{\cong}{\longrightarrow} \higherlim{\Or(\Ff^c)}{} R(P). \]
\end{introtheorem}

Hence the theorem above can be rephrased to say that $\K(|\Ll|\pcom)$ can be computed by stable elements. An important
component in the proof of this theorem is the fact that $p$--adic topological $K$--theory can also
be computed by stable elements, since this cohomology theory describes the stable behaviour
of complex vector bundles. In fact, we prove a stable elements formula for any generalized cohomology
theory.

\begin{introtheorem}
Given a $p$--local finite group $(S,\Ff,\Ll)$, there is an isomorphism 
\[ h^*(|\Ll |\pcom) \cong \higherlim{\Or(\Ff^c)}{} h^*(BP). \]
for any generalized cohomology theory $h^*$.
\end{introtheorem}

We note that this result was already established for cohomology with coefficients in $\F_p$ by Broto--Levi--Oliver 
in \cite{BLO1}. Motivated by the biset from \cite[Proposition 5.5]{BLO1}, Ragnarsson constructed in \cite{R1} a characteristic 
idempotent $\omega$ in the ring of stable self maps $\{ \Sigma^{\infty} BS, \Sigma^{\infty} BS \}$ associated to 
the saturated fusion system $\Ff$. This idempotent determines a stable summand $\B \Ff$ of $\Sigma^{\infty} BS$ which 
is homotopy equivalent to $\Sigma^{\infty} |\Ll| \pcom$ and it detects $\Ff$--stable maps into spectra, in the sense that a map
$f \colon \Sigma^{\infty} BS \to Z$ is $\Ff$--stable if and only if $ f \circ \omega \simeq f$.
The theorem mentioned above is a formal consequence of the existence of $\omega$ and its properties.  
 
In \cite{CL}, Castellana and Libman study the set of homotopy classes of maps between $p$--local finite groups. 
In particular they construct maps into the $p$--completed classifying space of a symmetric group which extend 
the regular representation of the Sylow subgroup. In that case, if one embeds the symmetric group as permutation matrices 
in a unitary group, we obtain a map which is a monomorphism in a $p$--local homotopic sense (see \fullref{Duality}
for the definition of homotopy monomorphism at the prime $p$).

Homotopy monomorphisms $|\Ll| \pcom \to BU(m) \pcom$ at the prime $p$ were also studied in \cite{CC} by the first two
authors of this article, where they are called unitary embeddings. Note that \cite{CC} has some overlap with this paper, 
but in this case the Sylow subgroup  is finite and that allows us to show the existence of unitary embeddings of $p$--local 
finite groups. Moreover, in \fullref{Duality} we prove the following improvement.

\begin{introtheorem}
Given a $p$--local finite group $(S,\Ff,\Ll)$, there exists a homotopy monomorphism $ |\Ll|\pcom \to BSU(n)\pcom$ 
whose homotopy fibre is a $\F_p$--finite space with Poincar\'e duality. 
\end{introtheorem}

The motivation for this theorem comes from duality properties of cohomology rings of finite groups. 
Benson--Carlson duality for cohomology rings of finite groups \cite{BC} shows that if the $\F_p$--cohomology ring of
a finite group is Cohen--Macaulay, then it is Gorenstein. On the other hand, the computation by Grbi\'c in \cite{Gr} 
of the $\F_2$--cohomology rings of the exotic $2$--local finite groups constructed in Levi--Oliver \cite{LO} shows that
these rings are Gorenstein. This suggested that an extension of Benson--Carlson duality should hold for 
$p$--local finite groups.

Dwyer, Greenlees and Iyengar \cite{DGI} viewed Benson--Carlson duality and other phenomena in the framework of ring spectra, 
showing that several dualities that appear in algebra and topology are particular cases of a more general
situation. From this point of view, Benson--Carlson duality is a consequence of the fact that the augmentation 
map $C^*(BG;\F_p) \to \F_p$ is Gorenstein in the sense of Definition 8.1 in \cite{DGI}. 

A careful analysis of this fact shows that it is a byproduct of having an injective
homomorphism $G \to SU(n)$. In this case $SU(n)/G$ satisfies Poincar\'e duality, but more importantly, only
mod $p$ Poincar\'e duality is needed to show the Gorenstein condition. This is the motivation for the theorem
mentioned above, and this is sufficient to prove a duality theorem for cohomology rings of $p$--local finite 
groups.

\begin{introtheorem}
Let $(S,\Ff,\Ll)$ be a $p$--local finite group. Then the augmentation map $C^*(|\Ll| \pcom; \F_p ) \to \F_p $ is 
Gorenstein in the sense of Dwyer--Greenlees--Iyengar \cite{DGI}. Therefore if $H^*(|\Ll| \pcom;\F_p)$ is Cohen--Macaulay, then it is Gorenstein. 
\end{introtheorem}

%%%%%%%%%%%%%%%%%%%%%%%%%%%%%%%%%%%%%%%%%%%%%%%%%%%%%%%%%%%%%%%%%%%
\subsection{Acknowledgements}
%%%%%%%%%%%%%%%%%%%%%%%%%%%%%%%%%%%%%%%%%%%%%%%%%%%%%%%%%%%%%%%%%%%

The authors are grateful to John Greenlees for suggesting that the existence 
of homotopy monomorphisms should have duality consequences on the cohomology ring.

%%%%%%%%%%%%%%%%%%%%%%%%%%%%%%%%%%%%%%%%%%%%%%%%%%%%%%%%%%%%%%%%%%%
\subsection{Acknowledgements of financial support}
%%%%%%%%%%%%%%%%%%%%%%%%%%%%%%%%%%%%%%%%%%%%%%%%%%%%%%%%%%%%%%%%%%%

The first two authors are partially supported by FEDER/MED (grants MTM2013-42293-P and MTM2016-80439-P),
and by SEP-CONACYT (grant 242186). The second author acknowledges financial support from the Spanish Ministry
of Economy and Competitiveness through the ``Mar\'ia de Maeztu'' Programme for Units of Excellence in R\&D (MDM-2014-0445).

%%%%%%%%%%%%%%%%%%%%%%%%%%%%%%%%%%%%%%%%%%%%%%%%%%%%%%%%%%%%%%%%%%%%%%%%%%%%%%%%
\section{Preliminaries on $p$--local finite groups}
\label{section:preliminaries}
%%%%%%%%%%%%%%%%%%%%%%%%%%%%%%%%%%%%%%%%%%%%%%%%%%%%%%%%%%%%%%%%%%%%%%%%%%%%%%%%

In this section, we recall the notion of a $p$--local finite group introduced by Broto, Levi and
Oliver in \cite{BLO1}. 
One of the ingredients is the concept of saturated fusion system introduced by Puig \cite{P}.
Given subgroups $P$ and $Q$ of $S$ we denote by $\Hom_S(P,Q)$ the set of group homomorphisms $P \to Q$ that
are conjugations by an element of $S$ and by $\Inj(P,Q)$ the set of monomorphisms from $P$ to $Q$.

\begin{definition}
A fusion system $\Ff$ over a finite $p$--group $S$ is a subcategory of the category of groups 
whose objects are the subgroups of $S$ and such that the set of morphisms $\Hom_{\Ff}(P,Q)$ between two 
subgroups $P$ and $Q$ satisfies the following conditions:
\begin{enumerate}
    \item[(a)] $\Hom_S (P,Q) \subseteq \Hom_{\Ff}(P,Q) \subseteq \Inj(P,Q)$ for all $P,Q \leq S$.
    \item[(b)] Every morphism in $\Ff$ factors as an isomorphism in $\Ff$ followed by an inclusion.
\end{enumerate}
\end{definition}
 
\begin{definition}
Let $\Ff$ be a fusion system over a $p$--group $S$.
\begin{itemize}
    \item We say that two subgroups $P,Q \leq S$ are $\Ff$--conjugate if they are isomorphic in $\Ff$.
    \item A subgroup $P\leq S$ is fully centralized in $\Ff$ if $|C_S (P)| \geq |C_S(P')|$ for all 
    $P' \leq S$ which are $\Ff$--conjugate to $P$.
    \item A subgroup $P\leq S$ is fully normalized in $\Ff$ if $|N_S (P)| \geq |N_S(P')|$ for all 
    $P' \leq S$ which are $\Ff$--conjugate to $P$.
    \item $\Ff$ is a saturated fusion system if the following conditions hold:
    \begin{enumerate}
        \item [(I)] Each fully normalized subgroup $P \leq S$ is fully centralized and
        the group $\Aut_S(P)$ is a $p$--Sylow subgroup of $\Aut_{\Ff}(P)$.
        \item [(II)] If $P \leq S$ and $\varphi \in \Hom _{\Ff}(P,S)$ are such that
        $\varphi P$ is fully centralized, and if we set
\[ N_\varphi = \{ g\in N_S(P) \mid \varphi c_g \varphi^{-1} \in \Aut_S(\varphi P) \}, \]
        then there is $\overline{\varphi} \in \Hom_{\Ff}(N_\varphi ,S)$ such that $\overline{\varphi}_{|P} =\varphi$.
    \end{enumerate}
\end{itemize}
\end{definition}

The motivating example for this definition is the fusion system of a finite group $G$. For a fixed Sylow 
$p$--subgroup $S$ of $G$, let $\Ff_S(G)$ be the fusion system over $S$ defined by setting $\Hom_{\Ff_S(G)}(P,Q)=\Hom_G(P,Q)$.
This is a saturated fusion system.

In the following definition we use the notation $\Out_{\Ff}(P)= \Aut_{\Ff}(P)/\Inn(P)$.

\begin{definition}
Let $\Ff$ be a fusion system over a $p$--group $S$.
\begin{itemize}
\item A subgroup $P \leq S$ is $\Ff$--centric if $P$ and all its $\Ff$--conjugates contain their $S$--centralizers.
\item A subgroup $P \leq S$ is $\Ff$--radical if $\Out_{\Ff}(P)$ is $p$--reduced, that is, if $\Out_{\Ff}(P)$ has no
proper normal $p$--subgroup.
\end{itemize}
\end{definition}

We will use $\Ff^c$ to denote the full subcategory of $\Ff$ whose objects 
are the $\Ff$--centric subgroups and $\Ff^{cr}$ for the full subcategory
of $\Ff$--centric, $\Ff$--radical subgroups.

The following theorem is a version of Alperin's fusion theorem for saturated fusion systems (Theorem A.10 in \cite{BLO1}).

\begin{theorem}
\label{theorem:Alperin}
Let $\Ff$ be a saturated fusion system over $S$. Then for each
isomorphism $P \to P'$ in $\Ff$, there exist
sequences of subgroups of $S$,
\[ P=P_0,P_1,\ldots,P_k=P' \qquad \textrm{and} \qquad Q_1,Q_2,\ldots,Q_k \]
and morphisms $\varphi_i \in \Aut_{\Ff}(Q_i)$ such that the following hold.
\begin{enumerate}
    \item $Q_i$ is fully normalized, $\Ff$--centric and $\Ff$--radical for each $i$.
    \item $P_{i-1},P_i \leq Q_i$ and $\varphi_i (P_{i-1})=P_i$ for each $i$.
    \item $\varphi= \varphi_k \circ \varphi_{k-1} \circ \cdots \circ \varphi_1$.
\end{enumerate}
\end{theorem}

The notion of a centric linking system is the extra structure needed in the definition of a $p$--local
finite group to obtain a classifying space which behaves like $BG \pcom$ for a finite group $G$.

\begin{definition}
\label{Linking}
Let $\Ff$ be a fusion system over a $p$--group $S$. A centric linking 
system associated to $\Ff$ is a category $\Ll$ whose objects are the 
$\Ff$--centric subgroups of $S$, together with a functor
\[ \pi \colon \Ll \longrightarrow \Ff ^c \]
and ``distinguished" monomorphisms $P \stackrel {\delta_P}{\longrightarrow} \Aut_{\Ll}(P)$ 
for each $\Ff$--centric subgroup $P \leq S$, which satisfy the following conditions.
\begin{enumerate}
        \item[(A)] $\pi$ is the identity on objects and surjective on morphisms. More precisely, for each
        pair of objects $P$, $Q$ in $\Ll$, $Z(P)$ acts freely on $\Mor_{\Ll}(P,Q)$ by composition (upon identifying
        $Z(P)$ with $\delta_P (Z(P)) \leq \Aut_{\Ll}(P)$) and $\pi$ induces a bijection
        \[ \Mor_{\Ll}(P,Q)/Z(P) \stackrel{\cong}{\longrightarrow} \Hom_{\Ff}(P,Q).  \]
        \item[(B)] For each $\Ff$--centric subgroup $P \leq S$ and each $g \in P$, the functor $\pi$ sends
        $\delta_P (g)$ to $c_g \in \Aut_{\Ff}(P).$
        \item[(C)] For each $f \in \Mor_{\Ll}(P,Q)$ and each $g \in P$, the following square commutes in $\Ll$
        \[
        \diagram
        P \rto^f \dto_{\delta_P (g)} & Q \dto^{\delta_Q(\pi(f)(g))} \\
        P \rto_f & Q 
        \enddiagram
        \]
\end{enumerate}
\end{definition}

\begin{definition}
A $p$--local finite group is a triple $(S,\Ff,\Ll)$, where $\Ff$ 
is a saturated fusion system over a finite $p$--group $S$ and 
$\Ll$ is a centric linking system associated to $\Ff$. The 
classifying space of the $p$--local finite group $(S,\Ff,\Ll)$ 
is the space $|\Ll | \pcom$.
\end{definition}

A theorem of Chermak \cite{C} (see also Oliver \cite{O}) states that every saturated fusion system admits a centric linking system  
and that it is unique up to isomorphism of centric linking systems. In particular, the classifying space of a $p$--local
finite group is determined up to homotopy equivalence by the saturated fusion system.

In many cases it is convenient to restrict the fusion system to certain subcategories. 
It was shown in Broto--Castellana-Grodal--Levi--Oliver \cite{BCGLO1} that one can consider certain full subcategories $\Ll_0$ 
of $\Ll$, such that the inclusion functor $\Ll_0\rightarrow \Ll$ induces a mod $p$ 
homotopy equivalence on nerves. In particular, one such is the full subcategory of 
$\Ff$--centric $\Ff$--radical subgroups of $S$.

Our main goal is to understand maps between classifying spaces. The key tool when 
studying maps between classifying spaces is the existence of mod $p$ homology 
decompositions. That is, we can reconstruct the classifying space of a $p$--local 
finite group up to $p$--completion as a homotopy colimit of classifying spaces of 
$p$--subgroups over the orbit category, $\Or(\Ff)$. The orbit category is the 
category whose objects are the subgroups of $S$ and whose morphisms are
\[ \Mor_{\Or(\Ff)}(P,Q)= \Rep_{\Ff}(P,Q)\stackrel{\emph{def}}{=}\Inn(Q)\setminus \Hom_{\Ff}(P,Q) . \]
Also, $\Or(\Ff^c)$ is the full subcategory of $\Or(\Ff)$ whose objects are the $\Ff$--centric subgroups of
$S$. The next proposition is part of Proposition 2.2 in \cite{BLO1}.

\begin{proposition}
Fix a saturated fusion system $\Ff$ and an associated centric
linking system $\Ll$, and let $\widetilde{\pi} \colon \Ll \to
\Or(\Ff^{c})$ be the projection functor. Let
\[ \widetilde{B} \colon \Or(\Ff^{c}) \longrightarrow \Top \]
be the left Kan extension along $\widetilde{\pi}$ of the constant
functor $\Ll \to \Top$ that sends every object to the one-point
topological space. Then $\widetilde{B}$ is a homotopy lifting of the homotopy functor $P \mapsto BP$, and
\[ |\Ll| \simeq \hocolim{\Or(\Ff^{c})}{} \, \widetilde{B}. \]
\end{proposition}

%%%%%%%%%%%%%%%%%%%%%%%%%%%%%%%%%%%%%%%%%%%%%%
\section{Complex representations of fusion systems}
\label{characters}
%%%%%%%%%%%%%%%%%%%%%%%%%%%%%%%%%%%%%%%%%%%%%%

Let $(S,\Ff,\Ll)$ be a given $p$--local finite group. In this section we study complex 
representations of the finite $p$--group $S$ which are compatible with the morphisms 
in the fusion category $\Ff$ in a sense that will be made precise. This description 
coincides with the one given in Section 3 of Cantarero--Castellana \cite{CC}. When dealing with a 
finite group $G$, this is described in Jackson \cite{J}. 
 
Given a finite group $G$, we denote by $\Rep_n(G)=\Rep(G,U(n))$ the set of isomorphism
classes of $n$--dimensional complex representations of $G$. If $ \rho$ is 
an $n$--dimensional representation of $G$, then $\chi_{\rho}$ denotes its associated 
character function.

\begin{definition}
Let $\Ff$ be a fusion system over $S$. An $n$--dimensional complex representation $\rho$ of $S$ is 
fusion-preserving if $\rho_{|P} = \rho_{|f(P)} \circ f$ in $\Rep_n(P)$ for any $P\leq S$ and 
any $ f \in \Hom_{\Ff}(P,S)$.
\end{definition}

We denote by $\Rep_n(\Ff)$ the set of isomorphism classes of $n$--dimensional complex representations of 
$S$ which are fusion-preserving. It is clear that if two representations of $S$ are isomorphic and one
of them is fusion-preserving, so is the other one. The following lemma gives an alternative description 
of this set.

\begin{lemma}
\label{respectfusion}
Let $\Ff$ be a fusion system over $S$. The composition 
\[ \higherlim{\Ff}{} \Rep _n (P) \subseteq \prod_{P\leq S} \Rep_n(P) \rightarrow \Rep_n(S) \]
of the inclusion and the projection map is injective and has image $\Rep_n(\Ff)$.
\end{lemma}

\begin{proof}
Given $Q \leq S$, an element $( \rho_P )_P$ in the limit must satisfy $\rho_Q = (\rho_S)_{|Q}$ 
since the inclusion $Q \to S$ is a morphism in $\Ff$. Therefore injectivity is clear. 
By definition, a representation $\rho$ of $S$ is fusion-preserving if and only if the element 
$( \rho_{|P} )_P$ belongs to $\higherlim{\Ff}{} \Rep_n (P)$, hence surjectivity follows. 
\end{proof}

\begin{remark}
\label{remark-character-fusion-preserving}
Since two representations are isomorphic if and only if their characters are equal, we can conclude that $\rho$ 
is fusion-preserving if and only if $\chi_{\rho}(g)$ equals $\chi_{\rho}(g')$ whenever there is a morphism $f$ in $\Ff$ 
such that $f(g)=g'$. Hence we can also think of $\Rep_n(\Ff)$ as the set of $n$--dimensional characters of $S$
which are fusion-preserving in that sense.
\end{remark}

\begin{example}
\label{example:regular}
Let $\reg$ be the regular representation of $S$. It has the property that $\chi_{\reg}(g)=0$ if $g \neq e$ and $\chi_{\reg}(e)=|S|$. 
Using \fullref{remark-character-fusion-preserving}, it is straightforward to check that $\reg \in \Rep_{|S|}(\Ff)$ for any fusion
system over $S$. Any trivial representation of $S$ is also fusion-preserving.
\end{example}

\begin{example}
\label{example:symmetric}
Let $\Sigma_3$ be the symmetric group on three letters, $S$ the subgroup generated
by $(1,2,3)$, which is a $3$-Sylow subgroup of $\Sigma_3$, and $\Ff = \Ff_S(\Sigma_3)$.
The trivial representation and the reduced regular representation of $S$ are fusion-preserving.
\end{example} 

When constructing fusion-preserving representations, it may be convenient to restrict to the
orbit category and to the family of centric or centric radical subgroups in $\Ff$ (see also
Remark 3.2 in Cantarero--Castellana \cite{CC}).

\begin{proposition}
\label{cr}
Let $\Ff$ be a saturated fusion system over $S$. Then
\[ \higherlim{\Ff^{cr}}{} \Rep _n (P) = \higherlim{\Ff^c}{} \Rep _n (P) = \Rep_n(\Ff). \]
Moreover
\[ \higherlim{\Or(\Ff^{cr})}{} \Rep _n (P) \cong \Rep_n(\Ff). \]
\end{proposition}

\begin{proof}
Since there are inclusions 
\[ \higherlim{\Ff}{} \Rep_n (P) \subseteq \higherlim{\Ff^c}{} \Rep _n (P) \subseteq \higherlim{\Ff^{cr}}{} \Rep _n (P), \]
it is enough to show that $\higherlim{\Ff^{cr}}{} \Rep_n (P) \subseteq \higherlim{\Ff}{} \Rep_n (P)$. 

Let $\rho \in \higherlim{\Ff^{cr}}{} \Rep_n (P)$. We have to show that $\rho_{|P} = \rho_{|P'} \circ \varphi$ in $\Rep_n(P)$ 
for any subgroups $P$, $P'$ of $S$ and any $\varphi \in \Hom_{\Ff}(P,P')$. Since this holds for inclusions, we may assume that $\varphi$ is
an isomorphism. By \fullref{theorem:Alperin}, there exists  a sequence of subgroups $P=P_0,P_1,\ldots,P_k=P'$ 
of $S$ and a sequence of $\Ff$--centric radical subgroups $Q_1,Q_2,\ldots,Q_k$ of $S$, with $P_{i-1},P_i \leq Q_i$, and a sequence 
of morphisms $\varphi_i \in \Aut_{\Ff}(Q_i)$ with $\varphi_i(P_{i-1})\leq P_i$ which factor $\varphi$. Since $\rho_{|Q_i} \circ \varphi_i$ 
and $\rho_{|Q_i}$ are isomorphic representations and $P_{i-1},P_i \leq Q_i$, the representations $\rho_{|P_i}$ and $\rho_{|P_{i-1}}$ 
are isomorphic to the respective restrictions of $\rho_{|Q_i}$. Therefore $\rho_{|P}$ is isomorphic to $\rho_{|P'} \circ \varphi$ in $\Rep_n(P)$.

Moreover, the functor $\Rep_n \colon \Ff^{cr} \to \Set$ factors through the orbit category $\Or(\Ff^{cr})$ 
since isomorphism classes of representations are fixed under inner automorphisms. Therefore the bijection just proved shows that 
\[ \higherlim{\Or(\Ff^{cr})}{} \Rep _n (P) \cong \Rep_n(\Ff). \qedhere \]
\end{proof}

The rest of this section is devoted to a construction which provides fusion-preserving representations out of representations of 
the Sylow subgroup. In the spirit of induction, the idea is to `induce' representations of the Sylow subgroup $S$ to fusion-preserving 
representations. 

The key tool which will allow us to do this construction is the specific $(S,S)$--biset $\Omega$ constructed in the proof of Proposition 5.5 of Broto--Levi--Oliver \cite{BLO1} for any saturated fusion system $\Ff$ over $S$. This biset $\Omega $ satisfies that 
$\Omega_{|(P,S)}$ and $\Omega_{|(\varphi,S)}$
are isomorphic as $(P,S)$--bisets for any  given $\varphi \in \Hom_{\Ff}(P,Q)$. Recall that $\Omega_{|(\varphi,S)}$ stands
for the $(P,S)$--biset whose underlying set is $\Omega$ with the same right action of $S$ and the left action of $P$ given by  
\[ p \cdot x = \varphi(p) x \]
where the action on the right hand side is the original right action of $S$. Hence there exists an isomorphism of 
$(P,S)$--bisets $\tau_{\varphi} \colon \Omega \to \Omega$ where the action on the source is via $\varphi$, that is, $\tau_{\varphi}(\varphi(p)x)= p \tau_{\varphi}(x)$.

In the remainder of the section we use $\C[X]$ to denote the complex vector space with the set $X$ as basis. If $S$ acts linearly on
the complex vector spaces $V$ and $W$ on the right and the left, respectively, we denote by $ V \otimes_S W$ their tensor
product as $\C S$--modules. Note that if $V$ has an action of $R$ on the left, $V \otimes_S W$ inherits a left $R$-action.

\begin{definition}
\label{definition-induction}
Let $\rho$ be a complex $n$--dimensional representation of the Sylow subgroup $S$ and consider the vector space $\C[\Omega] \otimes_S \C^n$, 
where $S$ acts on $\C^n$ via $\rho$. Define $\rho^{\Ff} \in \Rep_{n|\Omega/S|}(S)$ to be $\C[\Omega] \otimes_S \C^n$ with the left 
action of $S$ inherited from the left action of $S$ on $\Omega$. 
\end{definition}

\begin{proposition}
\label{proposition-induction} 
Let $\Ff$ be a saturated fusion system over $S$ and $\rho \in \Rep_n(S)$. Then $\rho^{\Ff} \in  \Rep_{n|\Omega/S|}(\Ff)$ and $\rho$ 
is a subrepresentation of $\rho^{\Ff}$.
\end{proposition}

\begin{proof} 
We need to show that for any $P \leq S$ and any $\varphi \in \Hom_{\Ff}(P,S)$, the representations $\rho^{\Ff}_{|P}$ and 
$\rho^{\Ff}_{|\varphi(P)} \circ \varphi$ of $P$ are isomorphic. The map $\tau_{\varphi} \colon \Omega \to \Omega$ considered
above induces a linear map $\C[\Omega] \to \C[\Omega]$, which we also denote by $\tau_{\varphi}$. Consider the map
\begin{gather*}
\overline{\tau}_{\varphi} \colon \C[\Omega]\otimes_S \C^n \longrightarrow \C[\Omega]\otimes_S \C^n \\
\qquad \enspace \quad x \otimes v \longmapsto \tau_{\varphi}(x) \otimes v.
\end{gather*}
The map $\overline{\tau}_{\varphi}$ is well defined and linear. It is bijective because $\tau_{\varphi}$ is bijective. Given
$x \in \Omega$, the equalities
\[ \overline{\tau}_{\varphi} (\varphi(p)(x \otimes v))= \overline{\tau}_{\varphi}(\varphi(p)x \otimes v) =\tau_{\varphi}( \varphi(p)x ) \otimes v= p\tau_{\varphi}(x) \otimes v = p \overline{\tau}_{\varphi}(x \otimes v). \]
show that $ \overline{\tau}_{\varphi}$ is an isomorphism from $\rho^{\Ff}_{|\varphi(P)} \circ \varphi$ to $\rho^{\Ff}_{|P}$.
Hence $\rho^{\Ff}$ is fusion-preserving.

Now we prove that $\rho$ is a subrepresentation of $\rho^{\Ff}$. Let $\{x_1, \ldots , x_r\} \subset \Omega$ be a set of orbit
representatives for $\Omega / S$. We can assume that $x_1$ is the element $[e,e] \in S \times _{(S,id)} S \subset \Omega$, where $e$ is the unit of $S$. It is clear that this element satisfies $s \cdot x_1 = x_1 \cdot s$ for all $s \in S$. 

Given a basis $\{v_1, \ldots , v_n\}$ of $\C^n$, let $V$ be the subspace of $\C[\Omega] \otimes_S \C^n$
generated by the elements $ x_1 \otimes v_i$ for $i=1, \ldots , n$. Note that $V$ is $S$--invariant since 
\[ s(x_1 \otimes v)=(s \cdot x_1) \otimes v= (x_1 \cdot s) \otimes v = x_1 \otimes \rho(s)(v). \]
Moreover, this shows that it is $S$--isomorphic to $\C^n$ with the action of $S$ via $\rho$. Therefore $\rho$ is a subrepresentation of $\rho^{\Ff}$.
\end{proof}

\begin{remark}
In particular, given a saturated fusion system $\Ff$ over $S$, any representation $\rho \in \Rep_n(S)$ 
is a subrepresentation of a fusion-preserving representation.
\end{remark}

%%%%%%%%%%%%%%%%%%%%%%%%%%%%%%%%%%%%%%%%%%%%%%%%%%%%%%%%%%%%%%%%%%%%%%
\section{Generalized cohomology theories of classifying spaces of $p$--local finite groups}
\label{section:generalized}
%%%%%%%%%%%%%%%%%%%%%%%%%%%%%%%%%%%%%%%%%%%%%%%%%%%%%%%%%%%%%%%%%%%%%%

This section contains the proof of \fullref{main-theorem-3} which relies strongly in the work of 
Ragnarsson \cite{R1} on the stable homotopy theory of fusion systems. We use $\{ E , F \}$ to denote 
the set of homotopy classes of maps between the spectra $E$ and $F$. 

Given a saturated fusion system $\Ff$ over $S$, Ragnarsson \cite{R1} constructs an idempotent in the 
ring of stable self-maps $\{ \Sigma^{\infty} BS,\Sigma^{\infty} BS \}$ associated to the saturated 
fusion system $\Ff$. More precisely, it is shown that there exists an idempotent $\omega$ in
$\{ \Sigma^{\infty} BS,\Sigma^{\infty} BS \}$ which is a $\Z \pcom$--linear combination of stable maps of the form $\Sigma^{\infty} B\varphi \circ \tr_P$, where $\tr_P \colon \Sigma^{\infty} BS \to \Sigma^{\infty} BP$ is the
stable transfer map and $\varphi \in \Hom_{\Ff}(P,S)$. This idempotent satisfies 
\[ \omega \circ \Sigma^{\infty} Bf \simeq \omega _{| \Sigma^{\infty} BP} \]
for any $P \leq S$ and any $f \in \Hom_{\Ff}(P,S)$. Here $\omega _{| \Sigma^{\infty} BP}$
denotes the composition of $w$ with the map $ j_P^S \colon \Sigma^{\infty} BP \to \Sigma^{\infty} BS$
induced by the inclusion of $P$ in $S$.

The homotopy type of the stable summand $\B \Ff$ of $\Sigma^{\infty} BS$ induced by 
$\omega$ coincides with the homotopy type of the classifying spectrum constructed by
Broto, Levi and Oliver in
\cite{BLO1} just after Proposition 5.5. Note that $\B \Ff$ is $p$--complete since it 
is a retract of $\Sigma^{\infty} BS$.

The spectrum $\B \Ff$ comes equipped with the structure map of the mapping telescope
$\sigma_{\Ff} \colon \Sigma^{\infty} BS \to \B \Ff$ and a transfer map $ t_{\Ff} \colon \B \Ff
\to \Sigma^{\infty} BS$ such that $ t_{\Ff} \circ \sigma_{\Ff} \simeq \omega$ and 
\mbox{$ \sigma_{\Ff} \circ t_{\Ff} \simeq \id$} (see Section 7 of \cite{R1}).

\begin{remark}
Let $\Ll$ be a centric linking system for $\Ff$ and $\Theta \colon BS \to | \Ll | \pcom $ the canonical inclusion 
induced by the structure morphism $\delta_S \colon S \to \Aut_{\Ll}(S)$ from \fullref{Linking}. 
Proposition 10.1 in \cite{R1} shows that there is a homotopy equivalence $ h \colon \B \Ff \to \Sigma^{\infty} |\Ll| \pcom $ 
such that $ h \circ \sigma_{\Ff} \simeq \Sigma^{\infty} \Theta$. 
\end{remark}

\begin{theorem} 
\label{main-theorem-3}
Let $h^*$ be a generalized cohomology theory.  Given a $p$--local finite group $(S,\Ff,\Ll)$, there is an isomorphism
\[ h^*(|\Ll |\pcom) \cong \higherlim{\Or(\Ff^c)}{} h^*(BP) . \]
\end{theorem}

\begin{proof}
Let $Y$ be a spectrum representing the corresponding reduced cohomology theory $\widetilde{h}^*$, so that we have
\[ h^n( \B \Ff )= \{ \B \Ff , Y \}_n = \{ \B \Ff, \Sigma^{-n} Y \} . \]
We will actually show that $ \{ \B \Ff , Z \} \cong \higherlim{\Or(\Ff^c)}{} \{ BP,Z \}$ for any spectrum $Z$. 

Let us denote by $f^*$ the image of a map $f$ of spectra under the functor $\{ -,Z \}$. We will show 
that $\im(\omega^*)$ is isomorphic to both $\higherlim{\Or(\Ff^c)}{}  \{\Sigma^{\infty} BP, Z\}$ and $ \{ \B \Ff , Z \} $.

Since $t_{\Ff}^* \sigma_{\Ff}^*$ is the identity, $\sigma_{\Ff}^*$ is injective and so its image is 
isomorphic to $\{ \B \Ff, Z \}$. On the other hand, $t_{\Ff}^*$ is surjective and therefore the image 
of $\omega^*$ equals the image of $\sigma_{\Ff}^*$. In particular, the image of $\omega^*$ is isomorphic 
to $\{ \B \Ff, Z \}$.

Now consider the map 
\begin{gather*}
\varphi \colon \im(\omega^*) \to \higherlim{\Or(\Ff^c)}{}  \{\Sigma^{\infty} BP, Z\} \\
\quad a \mapsto \left( a \circ j_P^S \right)_P 
\end{gather*}
It is well defined, because given $ f \colon P \to S$ in $\Or(\Ff^c)$ and $ a = \omega^*b \in \im(\omega^*)$ we 
have
\[ (\omega^* b) \circ \Sigma^{\infty} f = b \circ \omega \circ \Sigma^{\infty} f \simeq 
b \circ \omega \circ j_P^S = (\omega^* b) \circ j_P^S . \]
On the other hand, given $(b_P)_P$ in the inverse limit, the map $b_S$ is right $\Ff$--stable 
in the terminology of Definition 6.1 in \cite{R1}. Hence $b_S \circ \omega \simeq b_S$ by Corollary 6.4 in \cite{R1}
and therefore $\omega^*(b_S)=b_S$. So the projection 
\[\higherlim{\Or(\Ff^c)}{}  \{\Sigma^{\infty} BP, Z\} \to \{\Sigma^{\infty} BS, Z\} \]
factors through a map $\higherlim{\Or(\Ff^c)}{}  \{\Sigma^{\infty} BP, Z\} \to \im(\omega^*)$, which is the inverse of $\varphi$.
\end{proof}

\begin{remark}
Note that projection to the $\{ \Sigma^{\infty} BS, Z \}$--factor defines an isomorphism
\[ \higherlim{\Or(\Ff^c)}{}  \{\Sigma^{\infty} BP, Z\} \to \{ \Sigma^{\infty} BS, Z\}^{\Ff} \]
where $\{ \Sigma^{\infty} BS, Z\}^{\Ff}$ is the subgroup of stable elements, that is, maps $f$
such that \mbox{$ f \circ \Sigma^{\infty} \varphi \simeq f \circ j_P^S $} for any $\varphi \colon P \to S$
in $\Ff$. Hence this theorem also proved that there is an isomorphism
\[ \{ \Sigma^{\infty} |\Ll| \pcom, Z \} \cong \{ \Sigma^{\infty} BS, Z\}^{\Ff} \]
for any spectrum $Z$. In particular
\[ h^*(|\Ll |\pcom) \cong h^*(BS)^{\Ff} \]
for any generalized cohomology theory $h^*$.
\end{remark}

The following corollary will be particularly important in the next section. Recall 
that $p$--adic (periodic) topological $K$--theory $K^*(-;\Z \pcom)$ is the generalized 
cohomology theory associated to the spectrum determined by $ \Z \pcom \times BU \pcom$ 
(see Mitchell \cite{M} for instance).

\begin{corollary}
\label{corollary:Ktheory}
Let $(S,\Ff,\Ll)$ be a $p$--local finite group. Then 
\[ K^*(|\Ll| \pcom;\Z \pcom) \cong \higherlim{\Or(\Ff^c)}{} K^*(BP;\Z \pcom). \]
In particular, $K^*(|\Ll| \pcom;\Z \pcom)$ is torsion-free and concentrated in even degrees.
\end{corollary}

%%%%%%%%%%%%%%%%%%%%%%%%%%%%%%%%%%%%%%%%%%%%%%%%%%%%%%%%%%%%%%%%%%%%%%%%%%%%%%%%%%%
\section{Vector bundles over classifying spaces of $p$--local finite groups}
%%%%%%%%%%%%%%%%%%%%%%%%%%%%%%%%%%%%%%%%%%%%%%%%%%%%%%%%%%%%%%%%%%%%%%%%%%%%%%%%%%%

In this section we describe the Grothendieck group of complex vector bundles over the
classifying space of a $p$--local finite group in terms of the fusion-preserving characters 
of the Sylow subgroup $S$. Our main goal is to obtain a description for $p$--local finite groups
analogous to the one in Jackowski--Oliver \cite{JO}. We achieve this in \fullref{theorem:Grothendieck}
which will follow from \fullref{main-theorem-2}.
 
When dealing with a $p$--group $P$, Dwyer and Zabrodsky prove that there is  an isomorphism 
$\Vect(BP) \cong \Rep(P)$, where $\Vect(BP)$ and $\Rep(P)$ are the monoids 
of isomorphism classes of complex vector bundles over $BP$ and of complex finite 
dimensional representations of $P$, respectively. Therefore we obtain an isomorphism
of their Grothendieck groups $\K(BP) \cong R(P)$.

\begin{theorem}[\cite{DZ}]
\label{dwyer-zabrodsky}
There are natural bijections
\[ \Rep(P,U(m)) \to [BP,BU(m)] \to [BP,BU(m) \pcom] \]
given by sending a representation $\rho$ to $B\rho$ and composing with
the $p$--completion map \mbox{$BU(m) \to BU(m) \pcom$}. Moreover, 
the natural map 
\[ BC_{U(m)}(\rho(P))\to \Map(BP,BU(m))_{B\rho} \]
induces a homotopy equivalence 
\[ BC_{U(m)}(\rho(P))\pcom \to \Map(BP,BU(m)\pcom)_{B\rho}. \]
\end{theorem}

Recall that every $p$--local finite group $(S,\Ff,\Ll)$ comes equipped with a morphism \mbox{$\delta_S \colon S \to \Aut_{\Ll}(S)$} inducing 
a map $\Theta \colon BS \to |\Ll| \pcom$. Restriction defines a map \mbox{$ \Theta^* \colon [|\Ll| \pcom,BU(m)\pcom] \to [BS,BU(m)\pcom]$} 
which factors through the inverse limit  
\[ \psi_m \colon [|\Ll | \pcom, BU(m) \pcom] \longrightarrow \higherlim{\Or(\Ff^c)}{} [ BP, BU(m) \pcom ]. \] 

\fullref{dwyer-zabrodsky} allows us to give an algebraic description of the inverse limit
above in terms of fusion-preserving $m$-dimensional representations of $S$, since
\[ \higherlim{\Or(\Ff^c)}{} [BP,BU(m)] \cong \higherlim{\Or(\Ff^c)}{} \Rep_m(P) \cong \Rep_m(\Ff). \] 
Hence we may consider $\psi_m$ as a map 
\[ \psi_m \colon [|\Ll | \pcom, BU(m) \pcom] \longrightarrow \Rep_m(\Ff). \] 

We need to see how far $\psi_m$ stands from being injective and surjective. 
A general framework of obstruction theory to address this question 
has been developed by Wojtkowiak \cite{W}. There is a filtration $F_n|\Ll|\pcom$
of $|\Ll | \pcom$ induced by the skeletal filtration of the nerve of $\Or(\Ff^c)$
in such a way that an element in
\[ \higherlim{\Or(\Ff^c)}{} [BP,BU(m) \pcom ] \]
defines a map $F_1 | \Ll |\pcom \to BU(m) \pcom$ and the obstruction theory studies
how to extend it inductively to $F_n|\Ll| \pcom$. Note that the map $\psi_m$ constructed
above corresponds to restriction along the inclusion $F_1 | \Ll | \pcom \to  |\Ll| \pcom$.

In our case, given $ \rho $ in $ \Rep_m(\Ff) $, the obstructions for $\rho$ to be in the image or to have a unique preimage lie in higher limits of the functors
\begin{gather*}
F_i^{\rho} \colon  \Or (\Ff^c)^{\op} \to \Z _{(p)}\text{--Mod} \\
\qquad \qquad \qquad \qquad \qquad \qquad P  \mapsto \pi_i \left( \Map(\widetilde{B}P,BU(m) \pcom)_{\widetilde{B} \rho_{|P}} \right)
\end{gather*}
where we denote by $\widetilde{B} \rho_{|P}$ the composition of the map $ \widetilde{B}P \to BU(m)$ induced by $\rho$ and
the $p$--completion map $ BU(m) \to BU(m) \pcom$. More precisely, the obstruction to extend a map 
$ f \colon F_{n-1} | \Ll | \pcom \to BU(m) \pcom$ to $F_n | \Ll | \pcom$ without changing $f$ on $F_{n-2} | \Ll | \pcom $ is a class
\[ [E_n] \in \! \! \! \! \! \! \higherlim{\Or(\Ff^c)}{n+1} \pi_n \left( \Map(\widetilde{B}P,BU(m) \pcom)_{f_P} \right), \]
where $f_P$ is the restriction of $f$ to $\widetilde{B}P$. Similarly, the obstruction to extending a homotopy between two maps $f$ and $g$
which is already defined on $F_{n-1} | \Ll | \pcom $ without changing the homotopy on $F_{n-2} | \Ll | \pcom$ is a class
\[ [U_n] \in \higherlim{\Or(\Ff^c)}{n} \pi_n \left( \Map(\widetilde{B}P,BU(m) \pcom)_{f_P} \right). \]
More details can be found in Section 4 of Cantarero--Castellana \cite{CC}.

\begin{remark}
\label{DirectSum}
Given maps $ f \colon X \to BU(m) \pcom$ and $ g \colon X \to BU(n) \pcom$, 
we can define their Whitney sum $ f \oplus g$ to be the composition
\[ X \stackrel{(f,g)}{\longrightarrow} BU(m) \pcom \times BU(n) \pcom \stackrel{s}{\longrightarrow} BU(m+n) \pcom \]
where $ s \colon BU(m) \pcom \times BU(n) \pcom \to BU(m+n) \pcom$ is the map induced by 
the group homomorphism $U(m) \times U(n) \to U(m+n)$ that sends $(A,B)$ to the matrix 
with $A$ and $B$ as diagonal blocks.
\end{remark}

\begin{theorem}
\label{main-theorem-2}
Let $(S,\Ff,\Ll)$ be $p$--local finite group $(S,\Ff,\Ll)$ and let $\reg$ be the regular representation of $S$. The map 
\[ \psi_m \colon [|\Ll | \pcom, BU(m) \pcom] \longrightarrow \Rep_m (\Ff) \] 
has the following properties:
\begin{enumerate}
\item Given $\rho \in \Rep_m(\Ff)$, there exists a positive integer $M$ such that $\rho \oplus M{\reg}$ belongs
 to the image of $\psi_{m+M|S|}$. 
\item If $f_1$, $f_2 \colon |\Ll | \pcom \to BU(m) \pcom $ are such that $\psi_m(f_1)=\psi_m(f_2)$, then there exists 
\mbox{$ h \colon |\Ll | \pcom \to BU(n) \pcom$} for some $n$ such that $f_1 \oplus h \simeq f_2 \oplus h$ and $\psi_n(h)=N{\reg}$
for some positive integer $N$.
\end{enumerate}
\end{theorem}

\begin{proof}
Let $\rho$ be a fusion-preserving representation of $S$, that is, $\rho \in \Rep_m(\Ff)$ for some $m > 0$. 
For each $P \leq S$, let $\{\mu_1, \ldots, \mu_r\}$ be the set of all irreducible representations of $P$ 
and consider the decompositions
\[ \rho_{|P} = n_1 \mu_1 \oplus \cdots \oplus n_r \mu_r \]
\[ \reg_{|P} = k_1 \mu_1 \oplus \cdots \oplus k_r \mu_r \]
as sums of irreducible representations of $P$. Now
\[ \Map(\widetilde{B}P,BU(m) \pcom)_{\widetilde{B}\rho_{|P}} \simeq BC_{U(m)}(\rho(P)) \pcom \simeq \prod_{i=1}^r BU(n_i) \pcom, \]
where the first equivalence follows from Dwyer--Zabrodsky \cite{DZ} and the second equivalence is a consequence of 
Schur's Lemma. In particular, the space $\Map(\widetilde{B}P,BU(m) \pcom)_{\widetilde{B}\rho_{|P}}$ is simply 
connected.

The proof of Proposition 2.4 in \cite{JO} shows that the component of the constant map of $\Map(\widetilde{B}P,BU \pcom)$ 
satisfies
\[ \pi_i \left( \Map(\widetilde{B}P,BU \pcom)_0 \right) \cong K^{-i}(\widetilde{B}P) \otimes \Z \pcom \cong K_P^{-i}(\pt) \otimes \Z \pcom \]
for $i>0$ and $ R(P) \otimes \Z \pcom \cong K(\widetilde{B}P;\Z \pcom) $ by Lemma 2.1 in \cite{A}. Hence, if $z \colon BU(m) \pcom \to BU \pcom$ is the map induced by the inclusion $ U(m) \subseteq U$, we also have for $ i > 0$
\[ 
\pi_i \left( \Map(\widetilde{B}P,BU \pcom)_{z \circ \widetilde{B}\rho_{|P}} \right) \cong K^{-i}(\widetilde{B}P;\Z \pcom) \]
because all the components of $\Map(\widetilde{B}P,BU \pcom)$ are homotopy equivalent. Thus
the mapping space $\Map(\widetilde{B}P,BU \pcom)_{z \circ \widetilde{B}\rho_{|P}}$ is also simply connected. 
Therefore the first obstructions  $[E_1]$ to the respective extension problems vanish. Hence there are maps from
$ F_2 | \Ll | \pcom$ to $BU(m) \pcom$ and $BU \pcom$ that extend the maps $\widetilde{B}\rho_{|P}$
and $z \circ \widetilde{B}\rho_{|P}$, respectively. 
These results continue to hold if we replace $\rho$ by $\rho \oplus M {\reg}$ for any $M$.

By Corollary 3.4 in Broto--Levi--Oliver \cite{BLO1}, there exists a positive integer $N_{\Ff}$ such that
the higher limits of any functor $ G \colon \Or(\Ff^c) \to \Z _{(p)}$--Mod vanish above dimension $N_{\Ff}$.
We can assume $N_{\Ff} \geq 2$. Since the homotopy groups of the classifying spaces of complex unitary groups stabilize to the 
homotopy groups of $BU$, we can find a positive integer $ M>0$ such that the maps $BU(n_j+Mk_j) \to BU$
induce an isomorphism on the $i$th homotopy group for all $ i \leq N_{\Ff}$ and for all $j = 1, \ldots, r$. 
%There 
%is a homotopy equivalence 
%%
%\[ \Map(\widetilde{B}P,BU \pcom)_{z \circ \widetilde{B}\rho_{|P}} \simeq \prod_{i=1}^r BU \pcom \]
%%
%since we noticed above that both spaces have the same homotopy groups. 
If $i$ is even, postcomposition with the map $ z \colon BU(m+M|S|) \pcom \to BU \pcom$ induces a
commutative diagram
\[ 
\diagram
\pi_i \left( \prod \limits_{j=1}^r BU(n_j+Mk_j) \pcom \right) \dto^\cong \rto &  R(P) \otimes \Z \pcom \\
\pi_i(\Map(\widetilde{B}P,BU(m+M|S|)\pcom)_{\widetilde{B}(\rho \oplus M{\reg})_{|P}}) \rto & \pi_i(\Map(\widetilde{B}P,BU \pcom)_{z \circ \widetilde{B}(\rho \oplus M{\reg})_{|P}})\uto^{\cong}\\\enddiagram 
\]
where the top row corresponds to the map $ \prod BU(n_j + Mk_j) \pcom \to \prod BU \pcom$ induced by the inclusions \mbox{$U(n_j + Mk_j) \to U$}
and therefore it induces an isomorphism on the $i$th homotopy group for even $ i \leq N_{\Ff}$ (see \cite[Proposition  A.2]{O2}). A similar argument shows that the bottom isomorphism also holds for odd $i \leq N_{\Ff}$ since these homotopy
groups are all zero. Since the obstruction theory of Wojtkowiak \cite{W} is natural with respect to postcomposition, the obstructions of one extension problem are mapped 
to the other.  

Just as before, we have 
\[ \higherlim{\Or(\Ff^{c})}{i+1} \pi_i \left( \Map( BP,BU \pcom )_{z \circ \widetilde{B}(\rho \oplus M{\reg} )_{|P}} \right) \cong 
\! \! \! \! \! \! \higherlim{\Or(\Ff^{c})}{i+1} K^{-i}(BP;\Z \pcom). \]
This isomorphism sends the obstructions of the extension problem with $BU \pcom$
to the obstruction classes associated to the problem of existence of an element of 
$K^{-i}(| \Ll | \pcom ; \Z \pcom)$ that maps to the element of $\higherlim{\Or(\Ff^{c})}{} K^{-i}(BP;\Z \pcom)$ determined by the 
representation $ \rho \oplus M {\reg} $. \fullref{corollary:Ktheory} tells us that such an 
element exists. Therefore it is possible to construct a map $F_{N_{\Ff}+1} | \Ll | \pcom \to BU \pcom$ 
that extends the maps $z \circ \widetilde{B}(\rho \oplus M{\reg})_{|P}$. The obstructions of the extension
problem with $BU(m+M|S|)\pcom$ are mapped to the obstructions of the extension problem with $BU \pcom$
via an isomorphism, so these obstructions must vanish. Hence it is possible to construct a map 
\[ F_{N_{\Ff}+1} | \Ll | \pcom \to BU(m+M|S|)\pcom \]
that extends the maps $\widetilde{B}(\rho \oplus M{\reg})_{|P}$. Finally the lim$^{i+1}$--term of $F_i^{\rho \oplus M{\reg}}$ 
vanishes when $ i \geq N_{\Ff}$, so we can further extend it to a map $ f \colon | \Ll | \pcom \to BU(m+M|S|)\pcom$ which satisfies $\psi_{m+M|S|}(f) = \rho \oplus M {\reg}$.

Let $f_1$, $f_2 \colon | \Ll | \pcom \to BU(m) \pcom$ be such that $\psi_m(f_1) = \psi_m(f_2) = \rho $. By the first part
applied to the regular representation of $S$, there is some $M > 0$ such that the obstructions to existence vanish in each
step, hence 
\[ M {\reg} = \psi_{M|S|}(h) \]
for a certain map $ h \colon | \Ll | \pcom \to  BU(M|S|) \pcom$. Then we have
\[ \psi_{m + M|S|}(f_1 \oplus h) = \rho \oplus M {\reg} = \psi_{m + M|S|}(f_2 \oplus h). \]
Now we follow the same process as in the first part to construct a homotopy between $f_1 \oplus h$ and $f_2 \oplus h$. 
In this process, the first obstruction to uniqueness $[U_1]$ vanishes for the same reason, the obstructions up to
filtration level $N_{\Ff}$ vanish because \fullref{corollary:Ktheory} tells us that there is a unique element
in $K^{-i}(| \Ll | \pcom ; \Z \pcom)$ that maps to the element of $\higherlim{\Or(\Ff^{c})}{} K^{-i}(BP;\Z \pcom)$ determined by the 
representation $ \rho \oplus M {\reg} $. And the rest of obstructions vanish by Corollary 3.4 in Broto--Levi--Oliver \cite{BLO1}.
\end{proof}

\begin{remark}
\label{corollary:regular}
Note that the previous theorem shows that for any $p$--local finite group $(S,\Ff,\Ll)$, there exists 
a map $f \colon |\Ll| \pcom\to BU(n) \pcom$ such that $f_{|\widetilde{B}S} 
\simeq \widetilde{B}(M{\reg}) $ for some $M>0$.
\end{remark}

The maps $\psi_n$ from \fullref{main-theorem-2} assemble to define a map of monoids
\[ \coprod_{n \geq 0} [ |\Ll| \pcom , BU(n) \pcom] \to \coprod_{n \geq 0} \Rep_n(\Ff) \]
where the monoid structure on the first set is described in \fullref{DirectSum} and
on the right hand side is given by direct sum of representations. Therefore 
it induces a group homomorphism between their Grothendieck groups 
\[ \Psi \colon \K'(|\Ll| \pcom) \to R(\Ff) \]
such that the following diagram commutes
\[
\diagram
\K'(|\Ll| \pcom) \rto^{\Psi} \dto_{\res} & R(\Ff) \dto^{\res} \\
\K'(BS) \rto_{\cong} & R(S) 
\enddiagram
\]
The following lemma relates this group to the Grothendieck group
of complex vector bundles over $|\Ll| \pcom$.

\begin{lemma}
\label{VectorBundles}
The Grothendieck group of complex vector bundles over $|\Ll|\pcom$ 
is isomorphic to $\K'(|\Ll| \pcom)$.
\end{lemma}

\begin{proof}
Note that $p$--completion defines a map 
\[ [ | \Ll | \pcom , BU(n)] \to [ | \Ll | \pcom , BU(n) \pcom ]. \]
By \fullref{main-theorem-3}, we have $\widetilde{H}^k(| \Ll | \pcom;\Z[1/p]) = 0 $.
The space $BU(n)$ is simply connected, in particular nilpotent. Moreover, pointed homotopy classes of maps
into $BU(n)$ and $BU(n) \pcom$ coincide with unpointed homotopy classes. Therefore Theorem 1.5 in Miller \cite{Mi} 
shows that this map is a bijection. This defines an isomorphism of monoids
\[ \coprod_{n \geq 0} [ | \Ll | \pcom , BU(n)] \to \coprod_{n \geq 0} [ | \Ll | \pcom , BU(n) \pcom ], \]
hence their Grothendieck groups are isomorphic.
\end{proof}

Recall that $\K(X)$ denotes the Grothendieck group of complex vector bundles over $X$. Given
the result of the previous lemma, we will abuse the notation and use $\K(|\Ll| \pcom)$ for both
Grothendieck groups. We will now show that $\Psi$ is an isomorphism.

\begin{theorem}
\label{theorem:Grothendieck}
The map $\Psi \colon \K(|\Ll| \pcom) \to R(\Ff)$ is an isomorphism.
\end{theorem}

\begin{proof}
First we check that $\Psi$ is a monomorphism. Assume that we have two maps \mbox{$f \colon |\Ll| \pcom \to BU(n) \pcom$}
and $g \colon |\Ll| \pcom \to BU(m) \pcom$ such that $\Psi(f-g)=0$. Then we must have \mbox{$f_{|BS} - g_{|BS} = 0$}
in $\K(BS)$ and so there exists $ f' \colon BS \to BU(k) \pcom $ such that \mbox{$ f_{|BS} \oplus f' \simeq g_{|BS} \oplus f'$}. This implies $m = n$. Since $ [BS,BU(k)\pcom] \cong \Rep(S,U(k))$, we can assume that $f'$ is induced by a representation
$\rho$ of $S$. By \fullref{proposition-induction}, we can assume that $\rho$ is fusion-preserving, and
by \fullref{main-theorem-2}, that it belongs to the image of $\psi_k$. Hence we can take $f'$ to be the restriction
of a map $ t \colon |\Ll| \pcom \to BU(k) \pcom$ and we obtain $\psi_{m+k}(f \oplus t) = \psi_{m+k}(g \oplus t)$.
By \fullref{main-theorem-2}, there exists $h \colon |\Ll | \pcom \to BU(r) \pcom $ 
for some $r > 0$ such that $ f \oplus t \oplus h \simeq g \oplus t \oplus h$. Therefore $ f - g  = 0$ in $\K(|\Ll| \pcom)$.

Let $\chi \in R(\Ff)$, say $\chi = \rho_1 - \rho_2$, where both $\rho_1$ and $\rho_2$ are fusion-preserving. 
By \fullref{main-theorem-2} there exist positive integers $k_1$, $k_2$ and maps $f \colon |\Ll| \pcom \to BU(n_1) \pcom$ 
and $ g \colon | \Ll | \pcom \to BU(n_2) \pcom$ such that
\[ \psi_{n_1}(f)= \rho_1 + k_1 {\reg} \qquad \text{ and } \qquad \psi_{n_2}(g)= \rho_2 + k_2 {\reg}. \]
We can take $k_2$ big enough so that $(k_2 - k_1){\reg} \in \im(\psi_{(k_2-k_1)|S|})$. Then
\[ \chi = (\rho_1 + k_1 {\reg})-(\rho_2 + k_2 {\reg}) + (k_2 - k_1){\reg} \in \im(\Psi). \qedhere \]
\end{proof}

One could wonder at this point whether the Grothendieck construction of the monoid of 
fusion-preserving representations coincides with the inverse limit of representation rings
over the orbit category. The answer is given by the following proposition.

\begin{proposition}
\label{EquivalentDescription}
There is an isomorphism $ R(\Ff) \cong \higherlim{\Or(\Ff^c)}{} R(P)$.
\end{proposition}

\begin{proof}
Consider the map
\begin{gather*}
\varphi \colon R(\Ff) \to \higherlim{\Or(\Ff^c)}{} R(P) \\
\qquad \enspace \chi \mapsto (\res_P\chi)_P
\end{gather*}
where $\res_P$ is the composition $R(\Ff) \to R(S) \to R(P)$ of the respective restriction maps.
If $\chi = \rho_1 - \rho_2$, where $\rho_1$ and $\rho_2$ are fusion-preserving representations of $S$, 
then given $f \colon P \to S$ in $\Or(\Ff^c)$
\[ f^*\res_S(\rho_1 - \rho_2) = f^*(\rho_1) - f^*(\rho_2) = \res_P(\rho_1) - \res_P(\rho_2) = \res_P(\rho_1-\rho_2) \]
and so $\varphi$ is well defined. This map is clearly injective. On the other hand, given
an element $(\chi_P)_P$ of the inverse limit, consider $\chi_S = \alpha_1 - \alpha_2$. To
show surjectivity, it suffices to show that $\chi_S$ can be written as the formal difference
of two fusion-preserving representations. By \fullref{proposition-induction} and 
Maschke's Lemma, there is a representation
$\beta$ such that $\alpha_2 \oplus \beta = \alpha_2^{\Ff}$ and $\alpha_2^{\Ff}$ is fusion-preserving.
Then
\[ \chi_S = (\alpha_1 \oplus \beta) - \alpha_2^{\Ff} \]
and therefore, given $f \colon P \to S$ in $\Or(\Ff^c)$ we have
\[ \res_P(\alpha_1 \oplus \beta) - \res_P(\alpha_2^{\Ff}) = 
f^*(\alpha_1 \oplus \beta) - f^*(\alpha_2^{\Ff}) = f^*(\alpha_1 \oplus \beta) - \res_P(\alpha_2^{\Ff}). \]
That is, $\alpha_1 \oplus \beta$ is fusion-preserving as we wanted to show.
\end{proof}

\begin{remark}
Equivalently, the projection to the $R(S)$--component shows that we can also see 
$R(\Ff)$ as the subring of stable elements of $R(S)$.
\end{remark}

%%%%%%%%%%%%%%%%%%%%%%%%%%%%%%%%%%%%%%%%
\section{Duality}
\label{Duality}
%%%%%%%%%%%%%%%%%%%%%%%%%%%%%%%%%%%%%%%%

This section contains the proof of \fullref{Gorenstein}, that is, $C^*(|\Ll|;\F_p)\to \F_p$ is 
Gorenstein for any $p$--local finite group $(S,\Ff,\Ll)$. The motivation for this comes from extending
Benson--Carlson duality \cite{BC} to cohomology rings of $p$--local finite groups. This phenomenon
was already observed on the computation \cite{Gr} by Grbi\'c of the $\F_2$--cohomology rings of the exotic 
$2$--local finite groups constructed by Levi and Oliver in \cite{LO}. This suggested that an extension of Benson--Carlson 
duality should hold for $p$--local finite groups.

The strategy is to follow Example 10.3 of Dwyer--Greenlees--Iyengar \cite{DGI}, where it is shown that $C^*(BG;\F_p)\to \F_p$ 
is Gorenstein for any finite group $G$. The main ingredient in their proof is the existence of a complex faithful representation of $G$ into some $SU(n)$ 
such that $H^*(SU(n)/G;\F_p)$ is a Poincar\'e duality algebra. Our strategy is to mimic their proof using the 
existence of a homotopy monomorphism $|\Ll| \pcom \to BSU(m) \pcom$ at the prime $p$ from \fullref{main-theorem-1}. 
The first step in our proof of \fullref{Gorenstein} is to show that the mod $p$ cohomology of the 
homotopy fibre of such a map is a Poincar\'e duality algebra. This application of \fullref{main-theorem-1} 
was suggested to us by John Greenlees.

In order to prove \fullref{main-theorem-1}, we recall the notion of homotopy monomorphism at the
prime $p$ from Cantarero--Castellana \cite{CC}.

\begin{definition}
A connected pointed space $X$ is $B\Z/p$--null if the
pointed mapping space $\Map_*(B\Z/p,X)$ is contractible 
for any choice of basepoint in $X$. A map $f \colon X \to Y$ 
is called a homotopy monomorphism at $p$ if the homotopy
fibre of $f \pcom$ is $B\Z/p$--null.
\end{definition}

When the prime $p$ in question is clear, we will just write
homotopy monomorphism. Recall that a space is called $\F_p$-finite
if its $\F_p$-cohomology ring is a finite $\F_p$-vector space.

\begin{theorem}
\label{main-theorem-1}
There exists a homotopy monomorphism $|\Ll| \pcom \to BSU(m) \pcom$ for some \mbox{$m >0$}.
\end{theorem}

\begin{proof}
By \fullref{main-theorem-2}, there is a multiple of the regular
representation of $S$ in the image of some $\psi_n$. A preimage of 
this representation must be a homotopy monomorphism $|\Ll| \pcom \to BU(n) \pcom$
by Theorem 2.5 from Cantarero--Castellana \cite{CC}. Consider the map
\mbox{$BU(n) \pcom \to BSU(n+1) \pcom$} induced by the standard inclusion
of $U(n)$ in $SU(n+1)$. The homotopy fibre of this map is $(SU(n+1)/U(n)) \pcom$,
which is $\F_p$--finite, so it is a homotopy monomorphism by Proposition 2.2
from \cite{CC}. Moreover, since $(SU(n+1)/U(n)) \pcom$ is connected, Lemma 2.4 (c) from
\cite{CC} implies that the composition $|\Ll| \pcom \to BSU(n+1) \pcom$ of these 
two maps is a homotopy monomorphism.
\end{proof}

The restriction of a homotopy monomorphism $|\Ll| \pcom \to BSU(n) \pcom$ to $BS$ determines a faithful fusion-preserving representation $\rho \colon S \to SU(n)$. Since $\rho$ is injective, we abuse the notation and use $SU(n)/P$ to denote 
$SU(n)/\rho(P)$ for any $P \leq S$.

For what follows, we will consider again the $(S,S)$--biset $\Omega$ from Proposition 5.5 of Broto--Levi--Oliver \cite{BLO1}. 
Recall that $\Omega$ is a disjoint union of bisets of the form $ S\times_{(P,\varphi)}S$ with $P \leq S$ and $\varphi \colon P \to S$ in $\Ff$. Moreover, this set is $\Ff$--invariant in the sense that for each $P\leq S$ and each $\varphi \in \Hom_\Ff(P,S)$, $\Omega|_{(P,S)}$ and  $\Omega|_{(\varphi,S)}$ are isomorphic $(P,S)$--bisets. 

Each $(S,S)$--biset of the form $S \times_{(P,\varphi)} S$ with $\varphi \in \Hom_\Ff(P,S)$ induces an endomorphism 
of $H^*(SU(n)/S;\F_p)$ in the following way. The representation $\rho$ is fusion-preserving, hence there exists 
$A \in SU(n)$ such that $A\rho(p)A^{-1} = \rho(\varphi(p))$ for all $p \in P$. We define
\begin{gather*}
\overline{\varphi} \colon SU(n)/P \to SU(n)/S \\
\qquad \quad \enspace xP \mapsto AxA^{-1}S
\end{gather*}
This map is well defined and it does not depend on the choice of $A$ up to homotopy. To see this,
note that any two choices differ by an element in $C_{SU(n)}(\rho(P))$, which can be regarded as
the fibre of the determinant $ C_{U(n)}(\rho(P)) \to S^1$ over $1$. The centralizer in $U(n)$ is a 
product of unitary groups by Schur's lemma, hence path-connected. The determinant 
$C_{U(n)}(\rho(P)) \to S^1$ induces an epimorphism on the fundamental group since this centralizer is a product
of unitary groups and the determinant $U(m) \to S^1$ induces an isomorphism on the fundamental group. By the long exact 
sequence of homotopy groups, we conclude that $C_{SU(n)}(\rho(P))$ is path-connected. Therefore any two choices 
are connected by a path, which determines a homotopy between the two maps they would define.

Let $\tr_P^S$ be the transfer in cohomology with coefficients in $\F_p$ associated to the covering map $ i_P \colon SU(n)/P \to SU(n)/S$. Then we 
have an endomorphism of $H^*(SU(n)/S;\F_p)$ defined by
\[ \overline{\Omega}^* = \sum_{(P,\varphi) \in \Omega} n_{P,\varphi}\tr_P^S \overline{\varphi}^* \]
where $(P,\varphi) \in \Omega$ means $\Omega$ has a direct summand of the form $S \times_{(P,\varphi)} S$
and $n_{P,\varphi}$ is the number of times this summand appears.

The proof of \fullref{fibrestable} requires the following technical lemma.

\begin{lemma}
\label{StableUnitary}
The image of the endomorphism $\overline{\Omega}^*$ is isomorphic to the inverse limit over
$\Or(\Ff^c)$ of the cohomology groups $H^*(SU(n)/P;\F_p)$.
\end{lemma}

\begin{proof}
Consider the map
\begin{gather*}
g \colon \im(\overline{\Omega}^*) \to \higherlim{\Or(\Ff^c)}{} H^*(SU(n)/P;\F_p) \\
\! \! \! \! \! \! \! \! \! \! \! \! \! \! \! \! x \mapsto (i_P^* x)_P 
\end{gather*}
where $i_P \colon SU(n)/P \to SU(n)/S$ is the covering map introduced
above. Given $ \varphi \colon P \to Q$
in $\Or(\Ff^c)$ we have
\[ \varphi^* i_Q^* \overline{\Omega}^* = i_P^* \overline{\Omega}^* \]
because $\Omega$ is $\Ff$--invariant. Hence $g$ is well defined.  
Given an element $(x_P)_P$ in the inverse limit, we have
\[ \overline{\Omega}^*(x_S) = \sum_{(P,\varphi) \in \Omega} n_{P,\varphi} \tr_P^S \overline{\varphi}^*(x_S)=\sum_{(P,\varphi) \in \Omega} n_{P,\varphi} \tr_P^S i_P^*(x_S) \]
because $(x_P)_P$ belongs to the limit and
\[ \sum_{(P,\varphi) \in \Omega} n_{P,\varphi} \tr_P^S i_P^*(x_S)=\sum_{(P,\varphi) \in \Omega}n_{P,\varphi} [S:P]x_S = \frac{|\Omega|}{|S|}x_S =x_S \]
because the biset $\Omega$ satisfies $|\Omega|/|S| \equiv 1$ $($mod $p)$ (see Proposition 5.5 (c) in \cite{BLO1}). So $x_S$ belongs to the image of $\overline{\Omega}^*$ and $g(x_S) = (x_P)_P$. This shows that $g$ is surjective. Injectivity is clear.
\end{proof}

\begin{remark}
\label{Linearity}
Note that projection to $H^*(SU(n)/S;\F_p)$ gives an isomorphism
\[ \higherlim{\Or(\Ff^c)}{} H^*(SU(n)/P;\F_p) \to H^*(SU(n)/S;\F_p)^{\Ff} \]
to the direct summand of $\Ff$--stable elements in $H^*(SU(n)/S;\F_p)$, that is, elements $x$ such
that $\overline{\varphi}^* x = i_P^* x $ for any $\varphi \colon P \to S$ in $\Ff$. The proof of \fullref{StableUnitary}
shows that $\overline{\Omega}^*$ is the identity on this summand. Moreover, the endomorphism $\overline{\Omega}^*$ is $H^*(SU(n)/S;\F_p)^{\Ff}$--linear. Given elements $r$ in $H^*(SU(n)/S;\F_p)^{\Ff}$ and $x$ in $H^*(SU(n)/S;\F_p)$, then
\[ \tr_P^S \overline{\varphi}^*(rx) = \tr_P^S(i_P^*(r)\cdot \overline{\varphi}^*(x))=r \cdot \tr_P^S(\overline{\varphi}^*(x)), \]
where the first equality is due to the fact that $r$ is stable and the second equality holds by Frobenius reciprocity.
\end{remark}

\begin{proposition}
\label{fibrestable}
Let $F$ be the homotopy fibre of a homotopy monomorphism from
$|\Ll| \pcom$ to $BSU(n) \pcom$ at the prime $p$. Then there
is an additive isomorphism 
\[ H^*(F;\F_p) \cong \higherlim{\Or(\Ff^c)}{} H^*(SU(n)/P;\F_p) .\]
Moreover, $SU(n)/S$ is an orientable manifold and the cohomological fundamental 
class \mbox{$\omega_S \in H^*(SU(n)/S;\F_p)$} is a stable element.
\end{proposition}

\begin{proof}
For each $P\leq S$, consider the cohomological Serre spectral sequence $E_*(P)$ associated to the fibration $SU(n)\to SU(n)/P\to BP$. 
More precisely, this spectral sequence comes from the double complex 
\[ C^{r,s}(P)= \Hom_{\F_p[P]}(M_r, C^s(SU(n);\F_p)) \]
where $M_*$ is a free resolution of $\F_p$ as a $\F_p[S]$--module and $P$ acts on the singular cochains $C^*(SU(n);\F_p)$ via the linear representation 
$\rho$. There is a transfer 
\[ \tr_P^S \colon \Hom_{\F_p[P]}(M_r,C^s(SU(n);\F_p)) \to \Hom_{\F_p[S]}(M_r,C^s(SU(n);\F_p)) \]
given by $\tr_P^S(f)=\sum_{g_i\in X}g_ifg_i^{-1}$, where $X$ is a set of representatives for $S/P$. 
Each $\varphi \colon P \to S$ in $\Ff$ also induces a map
\[ \varphi^* \colon \Hom_{\F_p[S]}(M_r,C^s(SU(n);\F_p)) \to \Hom_{\F_p[P]}({^{\varphi}M_r},{^{\varphi}C^s(SU(n);\F_p))} \]
where ${^{\varphi}N}$ denotes $N$ with the action of $P$ through $\varphi$. Recall that $\varphi$ induces a
map $ \overline{\varphi} \colon SU(n)/P \to SU(n)/S$ given by $ xP \mapsto AxA^{-1}S$, where $A \in SU(n)$ is
such that \mbox{$ A \rho(p) A^{-1} = \rho(\varphi(p))$} for all $p \in P$. Given $f \colon SU(n) \to SU(n)$, let
us denote by $f^{\#}$ the induced endomorphism of $C^s(SU(n);\F_p)$. And given $B$ in $SU(n)$, we denote by $L_B \colon SU(n) \to SU(n)$
the map given by left multiplication by $B$. The map $ c_A \colon SU(n) \to SU(n)$ that sends $X$ to $AXA^{-1}$ 
defines a $P$--equivariant map
\[ c_A^{\#} \colon {^{\varphi}C^s(SU(n);\F_p)} \to C^s(SU(n);\F_p). \]
To see this, note that the action of $p \in P$ on ${^{\varphi}C^s(SU(n);\F_p)}$ is given by $L_{\rho(\varphi(p))}^{\#}$
and by $L_{\rho(p)}^{\#}$ on $C^s(SU(n);\F_p)$. And we have 
\[ c_A^{\#} L_{\rho(\varphi(p))}^{\#} = (L_{\rho(\varphi(p))} \circ c_A)^{\#} = (c_A \circ L_{\rho(p)})^{\#} = 
L_{\rho(p)}^{\#} c_A^{\#}, \]
where the second equality holds because
\[ L_{\rho(\varphi(p))}c_A(B) = \rho(\varphi(p)) ABA^{-1} = A \rho(p) A^{-1} A BA^{-1} = c_A L_{\rho(p)}(B). \]
Therefore $c_A^{\#}$ induces a map  
\[ \Hom_{\F_p[P]}({^{\varphi}M_r},{^{\varphi}C^s(SU(n);\F_p))} \to \Hom_{\F_p[P]}({^{\varphi}M_r},C^s(SU(n);\F_p)). \]
And we denote by $\widetilde{\varphi}^*$ the composition of $\varphi^*$ with this map, that is:
\[ \widetilde{\varphi}^* \colon \Hom_{\F_p[S]}(M_r,C^s(SU(n);\F_p)) \to \Hom_{\F_p[P]}({^{\varphi}M_r},C^s(SU(n);\F_p)). \] 

Since the maps $\tr_P^S$ and $\widetilde{\varphi}^*$ commute with the differentials, we obtain an endomorphism of the double complex $C^{*,*}(S)$
\[ \widetilde{\Omega}^* = \sum_{(P,\varphi)\in \Omega} n_{P,\varphi} \tr_P^S \widetilde{\varphi}^*. \]
Therefore, we have an induced endomorphism of each term $E_k(S)$ in the spectral sequence, which we denote by $\Omega_k^*$. 
By construction, the image of $\Omega_k^*$ consists of the stable elements, hence
\[ \im(\Omega_k^*) \cong \higherlim{\Or(\Ff^c)}{} E_k(P) \]
for all $k$.

Left multiplication of $S$ on $SU(n)$ via $\rho$ induces the action of $S$ on the cohomology groups of 
$SU(n)$ associated to the fibration $SU(n) \to SU(n)/S \to BS$. But left multiplication via $\rho$ factors 
through the left multiplication action of $SU(n)$ on itself and since $SU(n)$ is connected, the map given 
by left multiplication by an element of $SU(n)$ is homotopic to the identity map. Therefore the action of $S$ 
on $H^*(SU(n);\F_p)$ is trivial and so we have
\[ E_2^{r,s}(S) \cong H^r(BS;\F_p) \otimes H^s(SU(n);\F_p). \]
Note that by construction, $\Omega_2^*$ coincides with the morphism $[\Omega]$ from Proposition 5.5 of \cite{BLO1}
on $H^*(BS;\F_p)$ and it is the identity on the $H^s(SU(n);\F_p)$--factor since $c_A^*$ is the identity. In particular,
the choice of $A$ does not affect $\Omega_k^*$ for $k \geq 2$.

Consider now the cohomological Serre spectral sequence $E_*$ associated to the fibration $SU(n)\pcom \to F \to |\Ll|\pcom$.
The action of the fundamental group of $|\Ll| \pcom$ on \mbox{$H^*(SU(n) \pcom;\F_p) \cong H^*(SU(n);\F_p)$} factors through
the action of $S$, and so it is also trivial. Therefore the $E_2$--term is given by
\[ E_2^{r,s} \cong H^r(|\Ll| \pcom ;\F_p) \otimes H^s(SU(n);\F_p). \]
The map $ \Theta \colon BS \to |\Ll| \pcom$ induces a map of fibre sequences
\[
\diagram
SU(n) \pcom \rto \ddouble & (SU(n)/S) \pcom \rto \dto & BS \dto \rto & BSU(n) \pcom \ddouble \\
SU(n) \pcom \rto & F \rto & |\Ll| \pcom \rto & BSU(n) \pcom
\enddiagram
\]
which in turn induces a morphism of the associated Serre spectral sequences. Considering the maps involved
in this diagram, the corresponding morphism of spectral sequences \mbox{$E_2 \to E_2(S)$} is the morphism induced by the map $ \Theta \colon BS \to |\Ll| \pcom$
on the first factor of the tensor product and the identity on the second factor. Since the cohomology of a $p$--local 
finite group is computed by stable elements, the image of $E_2 \to E_2(S)$ is precisely $ \im([\Omega]) \otimes H^s(SU(n);\F_p)$,
which coincides with $\im(\Omega_2^*)$. And therefore $E_k \cong \im(\Omega_k^*)$ for all $k$.

Since $SU(n)$ is finite-dimensional, the spectral sequences $E_*$ and $E_*(S)$ collapse at a finite stage
and therefore
\[ E_{\infty} \cong \im(\Omega_{\infty}^*) \cong \higherlim{\Or(\Ff^c)}{} E_{\infty}(P). \]
Now $E_*$ converges to $H^*(F;\F_p)$ and $E_*(P)$ converges to $H^*(SU(n)/P;\F_p)$. The endomorphism $\Omega_{\infty}^*$ 
of $E_{\infty}(P)$ coincides with the one induced by $\overline{\Omega}^*$ from \fullref{StableUnitary} because the maps
$c_A$ induce the maps $ SU(n)/P \to SU(n)/S $ that take $xP$ to $AxA^{-1}S$. Therefore we have an isomorphism
of cohomology groups
\[ H^*(F;\F_p) \cong \higherlim{\Or(\Ff^c)}{} H^*(SU(n)/P;\F_p). \]

The action of $S$ on $SU(n)$ via $\rho$ is free and trivial on cohomology, hence $SU(n)/S$ is an orientable
manifold. The same holds for any $P \leq S$. It remains to show that the fundamental class 
$\omega_S \in H^N(SU(n)/S;\F_p)$ is stable, where $N$ is the dimension of $SU(n)$. For each $P < S$, 
the quotient $SU(n)/P \to SU(n)/S$ is a covering map of $p$--power index, hence the induced
map in the $N$th $\F_p$--cohomology group is zero. Therefore
\[ H^N(F;\F_p) \cong \higherlim{\Or(\Ff^c)}{} H^N(SU(n)/P;\F_p) \cong H^N(SU(n)/S;\F_p)^{\Aut_{\Ff}(S)}. \]
But the action of an element $\varphi \in \Aut_{\Ff}(S)$ on $SU(n)/S$ factors through the action of $SU(n)$
on $SU(n)/S$ by conjugation, and so it is trivial on cohomology. Therefore
\[ H^N(F;\F_p) \cong H^N(SU(n)/S;\F_p) \]
and in particular, the fundamental class of $SU(n)/S$ is stable.
\end{proof}

\begin{corollary}
\label{PoincareDuality}
$H^*(F;\F_p)$ is a Poincar\'e duality algebra.
\end{corollary}

\begin{proof}
Let $N$ be dimension of $SU(n)$. By \fullref{fibrestable} and \fullref{Linearity}, we know that  $H^*(F;\F_p)\cong H^*(SU(n)/S;\F_p)^{\Ff}$. In particular, the cohomology groups $H^i(F;\F_p)$ vanish for $i>N$. \fullref{fibrestable} also showed that
the fundamental class $\omega_S \in H^N(SU(n)/S;\F_p)$ is a stable element and $H^N(F;\F_p) \cong H^N(SU(n)/S;\F_p)$. 
Since $H^*(SU(n)/S;\F_p)$ is a Poincar\'e duality algebra, the following diagram defines a pairing for $H^*(SU(n)/S;\F_p)^{\Ff}$
\[
\diagram
H^i(SU(n)/S;\F_p)^{\Ff} \otimes H^{N-i}(SU(n)/S;\F_p)^{\Ff} \rto \dto & H^N(SU(n)/S;\F_p)^{\Ff} \ddouble  \\
H^i(SU(n)/S;\F_p) \otimes H^{N-i}(SU(n)/S;\F_p) \rto & H^N(SU(n)/S;\F_p) 
\enddiagram
\]
It remains to show that this pairing is non-singular. It is enough to show that for any $a$ in $H^i(SU(n)/S;\F_p)^{\Ff}$, 
there exists $b \in H^{N-i}(SU(n)/S;\F_p)^{\Ff}$ such that $a \smile b=\omega_S$. Since $ H^i(SU(n)/S;\F_p)$ is a Poincar\'e duality algebra, there exists $ b'\in H^{N-i}(SU(n)/S;\F_p)$ such that $a \smile b'=\omega_S$. Consider the element
$b=\overline{\Omega}^*(b')$ in $H^{N-i}(SU(n)/S;\F_p)^{\Ff}$. By \fullref{Linearity}, the morphism $\overline{\Omega}^*$ is $H^*(SU(n)/S;\F_p)^{\Ff}$--linear and so
\[ a \smile \overline{\Omega}^*(b') =\overline{\Omega}^*(a \smile b')=\overline{\Omega}^*(\omega_S)=\omega_S, \]
where the last equality holds because $\omega_S$ is a stable element.
\end{proof}

Recall that a map $R \to k$ of differential graded algebras is Gorenstein of shift $a$ if there
is a quasi-isomorphism $ \Hom_R(k,R) \sim \Sigma^a k$ of differential graded algebras over $k$
and the natural map
\[ \Hom_R(k,R) \otimes_{\End_R(k)} \Hom_R(k,k) \to \Hom_R(k,\Hom_R(k,R) \otimes_{\End_R(k)} k) \]
is a quasi-isomorphism of differential graded algebras over $\End_R(k)$. This is a particular
case of Definition 8.1 from Dwyer--Greenlees--Iyengar \cite{DGI}.

As a consequence of \fullref{PoincareDuality}, we obtain that $C^*(|\Ll| \pcom;\F_p)\to \F_p$ is Gorenstein. 
The proof follows the argument in \cite[Example 10.3]{DGI}, and we refer to this article for the relevant
notions which appear in this proof and the following results.

\begin{theorem}
\label{Gorenstein}
Let $(S,\Ff,\Ll)$ be a $p$--local finite group. Then the augmentation $C^*(|\Ll| \pcom;\F_p)\to \F_p$ is Gorenstein.
\end{theorem}

\begin{proof}
By \fullref{main-theorem-1}, there is a homotopy monomorphism $|\Ll| \pcom \to BSU(m) \pcom$ for some $m \geq 0$. 
Let $F$ be the homotopy fibre of this map. Note that $C^*(X \pcom;\F_p)$ is 
quasi-isomorphic to $C^*(X;\F_p)$ if $X$ is $p$--good. In particular,
this holds for $|\Ll| \pcom$, $BSU(m)$ and $SU(m)$.

By \fullref{fibrestable}, all the homology groups $H_k(F;\F_p)$ are finite-dimensional. 
And since $BSU(m) \pcom$ and $SU(m) \pcom$ are simply connected, the fundamental group of $F$ 
is isomorphic to the fundamental group of $|\Ll| \pcom$. This is a finite $p$--group by Proposition 
1.12 in \cite{BLO1}. Therefore $(F,\F_p)$ is of Eilenberg--Moore type.
 
Since $H^*(F;\F_p)$ is finite-dimensional, Remark 5.5 (2) in \cite{DGI} tells us that the 
augmentation $C^*(F;\F_p) \to \F_p$ is cosmall. By Remark 4.15 in \cite{DGI}, it is also proxy-small.
Now Proposition 8.12 from \cite{DGI} and \fullref{PoincareDuality} above imply that this augmentation
is Gorenstein.

Theorem 7.14 in \cite{Mc} shows that there is a quasi-isomorphism
\[ C^*(F;\F_p) \sim C^*(|\Ll| \pcom;\F_p) \otimes_{C^*(BSU(m) \pcom;\F_p)} \F_p. \]
By Section 10.2 in \cite{DGI}, the augmentation $C^*(BSU(m) \pcom;\F_p) \to \F_p$ is small and Gorenstein.
Since the augmentation $C^*(BSU(m) \pcom;\F_p) \to \F_p$ is small and the morphism $C^*(F;\F_p) \to \F_p$ is
proxy-small, the augmentation $C^*(|\Ll| \pcom;\F_p) \to \F_p$ is proxy-small by Proposition 4.18
in \cite{DGI}.

The fibration $ F \to |\Ll| \pcom \to BSU(m) \pcom$ is admissible in the sense of Dwyer--Wilkerson \cite{DW}, hence
$C^*(|\Ll| \pcom;\F_p)$ is small over $C^*(BSU(m) \pcom;\F_p)$ by Lemma 2.10 in \cite{DW}. The augmentation maps \mbox{$C^*(BSU(m) \pcom;\F_p) \to \F_p$} and $C^*(F;\F_p) \to \F_p$ are both Gorenstein, so we can use Proposition 8.10 from \cite{DGI} to conclude 
that $C^*(|\Ll| \pcom;\F_p) \to \F_p$ is Gorenstein. 
\end{proof}

\begin{corollary}
\label{Loops}
Let $(S,\Ff,\Ll)$ be a $p$--local finite group and let $\Omega (|\Ll| \pcom)$ denote the based loopspace
of $|\Ll| \pcom$. Then the augmentation $C_*(\Omega (|\Ll| \pcom);\F_p)\to \F_p$ is Gorenstein.
\end{corollary}

\begin{proof}
Since $(|\Ll|\pcom,\F_p)$ is of Eilenberg--Moore type, it is dc-complete (see Section 4.22 in \cite{DGI}).
We saw in the proof of the previous theorem that the augmentation $C^*(|\Ll|\pcom;\F_p) \to \F_p$ is
proxy-small and Gorenstein. By Proposition 8.5 in \cite{DGI}, we conclude that \mbox{$C_*(\Omega (|\Ll| \pcom);\F_p)\to \F_p$} is Gorenstein.
\end{proof}

Moreover, as in \cite[Example 10.3]{DGI}, we get other interesting consequences, such as the existence of a local cohomology spectral sequence.

\begin{corollary}
\label{LocalCohomology}
Let $(S,\Ff,\Ll)$ be a $p$--local finite group. There is a spectral sequence
\[ E_{i,j}^2 = H_I^{-i}(H^*(|\Ll| \pcom;\F_p))_j \Rightarrow H_{i+j}(|\Ll|\pcom;\F_p) \]
where $I$ is the ideal of elements of positive dimension.
\end{corollary}

\begin{proof}
Since $C^*(|\Ll| \pcom;\F_p)$ is coconnective and connected, it follows by Remark 3.17 in \cite{DGI}
that $C_*(|\Ll| \pcom;\F_p)$ is $\F_p$--cellular over $C^*(|\Ll| \pcom;\F_p)$. Since the $\F_p$--cohomology
of $|\Ll| \pcom$ is Noetherian, it follows from Proposition 9.3 in \cite{DGI} that there is a spectral
sequence
\[ E_{i,j}^2 = H_I^{-i}(H^*(|\Ll| \pcom;\F_p))_j \Rightarrow H_{i+j-a}(|\Ll|\pcom;\F_p) \]
where $I$ is the ideal of elements of positive dimension and $C^*(|\Ll| \pcom;\F_p) \to \F_p$ is
Gorenstein of shift $a$. By \fullref{PoincareDuality}, $F$ is a Poincar\'e duality algebra 
of the same dimension of $SU(m)$, so the shift of $C^*(BSU(m) \pcom;\F_p) \to \F_p$ and $C^*(F;\F_p) \to \F_p$
coincide. By Propositions 8.6 and 8.10 in \cite{DGI}, the shift of  $C^*(|\Ll| \pcom;\F_p) \to \F_p$
is zero.
\end{proof}

Recall that a graded commutative Noetherian local ring $R$ with maximal ideal $\m$ and residue field $k$ is 
Cohen--Macaulay if its local cohomology is concentrated in one degree. In this case, $R$ is Gorenstein if the 
local cohomology in this degree is isomorphic to $\Hom_k(R,k)$ (see Greenlees--Lyubeznik \cite{GL} for instance).

The local cohomology spectral sequence has structural implications on the cohomology of $|\Ll| \pcom$. For example, if $H^*(|\Ll| \pcom ;\F_p)$ 
is Cohen--Macaulay, then the spectral sequence collapses to give an isomorphism 
\[ H_I^r(H^*(|\Ll| \pcom;\F_p)) \cong H_*(|\Ll| \pcom;\F_p) \cong \Hom_{\F_p}(H^*(|\Ll| \pcom;\F_p),\F_p) \]
and so it is Gorenstein (see Greenlees \cite{G} and Greenlees--Lyubeznik \cite{GL}).

\begin{example}
Some important examples of exotic $2$--local finite groups were
constructed in Levi--Oliver \cite{LO} (see also Benson \cite{Be}), motivated by the work
of Solomon \cite{S} of classifying all finite simple groups whose $2$--Sylow
subgroups are isomorphic to those of the Conway group $\Co_3$.
They construct a $2$--local finite group $(S,\Ff_{\Sol}(q),\Ll^c_{\Sol}(q))$
over a $2$--Sylow subgroup $S$ of $\Spin_7(q)$ for any odd prime power $q$. The 
$\F_2$--cohomology of these $2$--local finite groups was computed by
Grbi\'c \cite{Gr} to be
\[ H^*(|\Ll^c_{\Sol}(q)| ^{\wedge}_2;\F_2) \cong \F_2[u_8,u_{12},u_{14},u_{15},y_7,y_{11},y_{13}]/I \]
where $I$ is the ideal generated by the polynomials
\begin{gather*}
y^2_{11} + u_8 y^2_7 + u_{15} y_7, \\
y_{13}^2 + u_{12} y^2_7 + u_{15}y_{11}, \\
y_7^4 + u_{14} y_7^2 + u_{15} y_{13}.
\end{gather*}
In fact, Proposition 1 in \cite{Gr} shows that $H^*(|\Ll^c_{\Sol}(q)| ^{\wedge}_2;\F_2)$ is
a finitely generated free $\F_2[u_8,u_{12},u_{14},u_{15}]$-module. Therefore the cohomology
ring is Cohen--Macaulay (see Definition 5.4.9 and Theorem 5.4.10 in Benson \cite{Be2}). Hence our 
arguments above imply that it must be Gorenstein.

In this particular case we can actually deduce that it is Gorenstein from the computation. The quotient of 
$H^*(|\Ll^c_{\Sol}(q)| ^{\wedge}_2;\F_2)$ by the ideal generated by the polynomial subring $\F_2[u_8,u_{12},u_{14},u_{15}]$ is the graded ring
\[ \F_2[y_7,y_{11},y_{13}]/(y_{11}^2,y_{13}^2,y_7^4) \]
which is a Poincar\'e duality algebra. By Proposition I.1.4 and the Remark on the same page
of Meyer--Smith \cite{MS}, we can conclude that $H^*(|\Ll^c_{\Sol}(q)| ^{\wedge}_2;\F_2)$ is Gorenstein.
\end{example}

%%%%%%%%%%%%%%%%%%%%%%%%%%%%%%%%%%%%%%%%%%%%%%%%%%
%%%%%%%%%%%%%%%%%%% REFERENCES %%%%%%%%%%%%%%%%%%%
%%%%%%%%%%%%%%%%%%%%%%%%%%%%%%%%%%%%%%%%%%%%%%%%%%

\bibliographystyle{amsplain}

\bibliography{mybibfile}

%%%%%%%%%%%%%%%%%%%%%%%%%%%%%%%%%%%%%%

\end{document}